\documentclass[11pt,a4paper]{article}

\usepackage{amsfonts}
\usepackage{amsthm}
\usepackage{amsmath}
\usepackage{amssymb}
\usepackage{amscd}
\usepackage{mathrsfs}
\usepackage[mathscr]{eucal}
\usepackage{graphicx}
\usepackage{epstopdf} 
\usepackage{pict2e}
\usepackage{epic}
\usepackage{xcolor}
\usepackage{float}
\usepackage{mathtools}
\usepackage{centernot}

\usepackage[margin=1cm, font=small, labelfont=bf]{caption}

\usepackage{cite}
\usepackage{hyperref}
\hypersetup{colorlinks=true, urlcolor= black, linkcolor=black, citecolor=blue}

\usepackage[utf8]{inputenc}
\usepackage[T1]{fontenc}
\usepackage{lmodern}
\usepackage{dsfont}
\usepackage{indentfirst}

\usepackage{hyperref}   
\hypersetup{
    linktoc=page,
    linkcolor=blue,          
    citecolor=blue,        
    filecolor=blue,      
    urlcolor=cyan,
   colorlinks=true           
}

\numberwithin{equation}{section}

\theoremstyle{plain}
\newtheorem{Thm}{Theorem}[section]
\newtheorem{Lem}[Thm]{Lemma}

\newtheorem{Coro}[Thm]{Corollary}
\newtheorem{Prop}[Thm]{Proposition}

\theoremstyle{definition}
\newtheorem{Def}[Thm]{Definition}

\newtheorem{Rem}[Thm]{Remark}

\usepackage[margin=2.3cm]{geometry}

\definecolor{darkgreen}{rgb}{0,0.6,0.05}

\newcommand{\connect}{\xleftrightarrow}

\title{Sharp bounds on the half-space two-point function for high-dimensional Bernoulli percolation}

\begin{document}

\author{Romain Panis\footnotemark[1]\footnote{Universit\'e Claude Bernard Lyon 1, Villeurbanne, France, \url{panis@math.univ-lyon1.fr}, \url{schapira@math.univ-lyon1.fr}}\:, Bruno Schapira\footnotemark[1]}

\maketitle

\abstract{We consider Bernoulli percolation on $\mathbb Z^d$ with $d>6$. We prove an up-to-constant estimate for the critical two-point function restricted to a half-space. This completes previous results of Chatterjee and Hanson (Commun. Pure Appl. Math., 2021), and Chatterjee, Hanson, and Sosoe (Commun. Math. Phys., 2023), and solves a question asked by Hutchcroft, Michta, and Slade (Ann. Probab., 2023).
}

\section{Introduction}\label{sec:intro}
Let $\mathcal G=(\mathcal V,\mathcal E)$ be a graph with vertex set $\mathcal V$ and edge set $\mathcal E$. If $x,y\in \mathcal V$, we write $x \sim y$ to say that $\{x,y\}\in \mathcal E$. We consider Bernoulli percolation on $\mathcal G$. Given $p\in [0,1]$, we construct a random subgraph of $\mathcal G$ by independently keeping (resp. deleting) each edge of $\mathcal E$ with probability $p$ (resp. $1-p$). The associated measure is denoted by $\mathbb P_p$. We focus on the following examples of graphs $\mathcal G$: for $d\geq 2$,
\begin{enumerate}
	\item[$(i)$] Nearest-neighbour model: $\mathcal V=\mathbb Z^d$ and $\mathcal E=\{\{x,y\}: \Vert x-y\Vert_1=1\}$ where $\Vert \cdot\Vert_1$ is the $\ell^1$ norm on $\mathbb R^d$;
	\item[$(ii)$] Spread-out model with \emph{spread} parameter $L\geq 1$: $\mathcal V=\mathbb Z^d$ and $\mathcal E=\{\{x,y\}: \Vert x-y\Vert_1\leq L\}$.
\end{enumerate}
It is well-known (see for instance \cite{GrimmettPercolation1999}) that the model undergoes a non-trivial phase transition as the parameter $p$ varies:  one has $p_c\in(0,1)$ with 
\begin{equation}
	p_c:=\inf\{p\in [0,1]: \mathbb P_p[0\connect{}\infty]>0\},
\end{equation}
where $\{0\connect{}\infty\}$ is the event that the origin lies in an infinite connected component. 

In this paper, we study \emph{high-dimensional} percolation, meaning that we work in dimensions $d>6$. This corresponds to the (conjectured) \emph{mean-field} regime of the model. We refer to \cite{GrimmettPercolation1999,SladeSaintFlourLaceExpansion2006,panis2024applications,hutchcroft2025dimension} and references therein for more information on the particular role played by the dimension $d=6$ in percolation theory. We investigate properties of the critical measure $\mathbb P=\mathbb P_{p_c}$. A fundamental quantity in its analysis is the so-called (restricted) \emph{two-point function}, which is defined as follows: if $A\subset \mathbb Z^d$ and $x,y\in \mathbb Z^d$,
\begin{equation}
	\tau_{A}(x,y):=\mathbb P[x\connect{A\:}y],
\end{equation}
where $\{x\connect{A\:}y\}$ is the event that there exists an \emph{open path} (i.e.\ a path made of edges that were kept) fully contained in $A$ which connects $x$ and $y$. When $A=\mathbb Z^d$, we drop it from the above notation.

The starting point in the study of high-dimensional Bernoulli percolation is the following estimate on the critical two-point function: for every $x,y\in \mathbb Z^d$,
\begin{equation}\tag{$*$}\label{eq: 2pt full space estimate}
	\tau(x,y)\asymp \frac{1}{1+|x-y|^{d-2}},
\end{equation}
where $\asymp$ means that the ratio of the two quantities is bounded away from $0$ and infinity by two constants which only depend on $d$ (and potentially the spread parameter $L$), and where $|\cdot|$ denotes the $\ell^\infty$ norm on $\mathbb R^d$. The \emph{lace expansion} approach developed by Brydges and Spencer \cite{BrydgesSpencerSAW} (see \cite{SladeSaintFlourLaceExpansion2006} for a review) has been successfully implemented to derive a more precise version of \eqref{eq: 2pt full space estimate} for nearest-neighbour percolation in dimensions $d>10$ \cite{HaraSlade1990Perco,HaraDecayOfCorrelationsInVariousModels2008,FitznervdHofstad2017Percod10}, and sufficiently spread-out percolation (i.e. $L\gg 1$) in dimensions $d>6$ \cite{HaraSladevdHofstad2003PercoSO}. An alternative proof of \eqref{eq: 2pt full space estimate} in the latter setting has recently been obtained in \cite{DumPan24Perco}.

We are interested in the behaviour of the critical two-point function restricted to the half-space $\mathbb H:=\{x=(x_1,\dots,x_d) \in \mathbb Z^d :  x_1\ge 0\}$. In this setting, the main difficulty comes from the lack of full translation invariance. Nevertheless, several partial results have been obtained. The first set of results goes back to \cite{ChatterjeeHanson}. 
\begin{Prop}[\hspace{1pt}{\cite{ChatterjeeHanson}}]\label{prop:chatterjeehanson} Let $d>6$ and assume that \eqref{eq: 2pt full space estimate} holds. Then, for every $K\geq 1$, there exist $c,C>0$ such that the following holds:
\begin{enumerate}
	\item[$(a)$] For every $x,y\in \mathbb H$ which satisfy $|x-y|\leq K\min(x_1,y_1)$,
	\begin{equation}\label{CH1}
		\frac{c}{1+|x-y|^{d-2}}\leq \tau_{\mathbb H}(x,y)\leq \frac{C}{1+|x-y|^{d-2}}.
	\end{equation}
	\item[$(b)$] For every $x,y\in \mathbb H$ which satisfy $x_1=0$ and $|x-y|\leq Ky_1$,
	\begin{equation}\label{CH2}
		\frac{c}{1+|x-y|^{d-1}}\leq \tau_{\mathbb H}(x,y)\leq \frac{C}{1+|x-y|^{d-1}}.
	\end{equation}
	\item[$(c)$] For every $x,y \in \mathbb H$ which satisfy $x_1=y_1=0$,
	\begin{equation}\label{CH3}
		\frac{c}{1+|x-y|^{d}}\leq \tau_{\mathbb H}(x,y)\leq \frac{C}{1+|x-y|^{d}}.
	\end{equation}
\end{enumerate}
\end{Prop}
This result identifies three different regimes of decay for $\tau_{\mathbb H}(x,y)$. However, Proposition \ref{prop:chatterjeehanson} does not give any information on how the two-point function interpolates between these regimes. Partial steps in this direction where taken in the subsequent work \cite{ChatterjeeHansonSosoe2023subcritical}. Below, we let $\mathbf{e}_1=(1,0,\ldots,0)$.
\begin{Prop}[\hspace{1pt}{\cite{ChatterjeeHansonSosoe2023subcritical}}]\label{prop:chatterjeehansonsosoe} Let $d>6$ and assume that \eqref{eq: 2pt full space estimate} holds. Then, there exist $c,C>0$ such that the following holds:
\begin{enumerate}
	\item[$(a)$] For every $m\geq 1$, and every $x\in \mathbb H$, 
	\begin{equation}\label{CHS1}
		\tau_{\mathbb H}(x,m\mathbf{e}_1)\leq C\frac{1+m}{1+|x-m\mathbf{e}_1|^{d-1}}.
	\end{equation}
	\item[$(b)$] For every $m\geq 1$, and every $x\in \mathbb H$, if $x_1\geq \tfrac{1}{2}|x|$ and $|x|\geq 4m$, then
	\begin{equation}\label{CHS2}
		\tau_{\mathbb H}(x,m\mathbf{e}_1)\geq c\frac{1+m}{1+|x-m\mathbf{e}_1|^{d-1}}.
	\end{equation}
\end{enumerate}
\end{Prop}
\begin{Rem} Some of the estimates stated in Propositions \ref{prop:chatterjeehanson} and \ref{prop:chatterjeehansonsosoe} have recently been derived in the context of spread-out percolation \cite{DumPan24Perco} (and also in the context of the weakly-self avoiding walk model \cite{DumPan24WSAW}) using very different methods.
\end{Rem}

The estimates of Proposition \ref{prop:chatterjeehansonsosoe} are inefficient in the situation were both $x$ and $y$ lie near the boundary of $\mathbb H$. By analogy with Green function estimates (see e.g \cite{LawlerLimicRandomWalks2010}), Hutchcroft, Michta, and Slade \cite[Remark~3.4]{HutchcroftMichtaSladePercolationTorusPlateau2023} conjectured a behaviour for $\tau_{\mathbb H}(x,y)$ in the regime where $\max(x_1,y_1)\leq|x-y|$. Our main result is a proof of their conjecture. It provides a sharp (up-to-constant) estimate on $\tau_{\mathbb H}(x,y)$ for every $x,y\in \mathbb H$. We will need the following notation: if $x,y\in \mathbb H$, we let $r_{x,y}:=\min(x_1,|x-y|)$.
\begin{Thm} \label{thm:main} Let $d>6$ and assume that \eqref{eq: 2pt full space estimate} holds. Then, there exist $c,C>0$ such that, for every $x,y\in \mathbb H$, \begin{equation} 
c\frac { (1+ r_{x,y}) \cdot (1+ r_{y,x})}{1+ |x-y|^d}\leq \tau_{\mathbb H} (x,y )\leq C\frac { (1+ r_{x,y}) \cdot (1+ r_{y,x})}{1+ |x-y|^d}.
\end{equation}
\end{Thm} 
\begin{Rem} It is interesting to compare our result with Proposition \ref{prop:chatterjeehansonsosoe}. The latter result can be rephrased as follows: if $d>6$ and \eqref{eq: 2pt full space estimate} holds, then there exist $c,C>0$ such that, for every $x,y\in \mathbb H$,
\begin{equation}
	\tau_{\mathbb H}(x,y)\leq C\frac{1+\min(r_{x,y},r_{y,x})}{1+|x-y|^{d-1}},
\end{equation}
and, assuming (for instance) that $x_1\geq \frac{1}{2}|x|$ and $|x|\geq 4|y|$, 
\begin{equation}
		\tau_{\mathbb H}(x,y)\geq c\frac{1+y_1}{1+|x-y|^{d-1}}.
\end{equation}
Therefore, Theorem \ref{thm:main} corresponds to Proposition \ref{prop:chatterjeehansonsosoe} in the regime where $r_{x,y}\asymp |x-y|$.
\end{Rem}
As an immediate corollary of Theorem \ref{thm:main}, we obtain an alternative (short and easy) proof of \cite[Proposition~3.1]{HutchcroftMichtaSladePercolationTorusPlateau2023} (which motivated \cite[Remark~3.4]{HutchcroftMichtaSladePercolationTorusPlateau2023}).
For every $n\in \mathbb Z$, let $\mathbb H_n:=\mathbb H-n\mathbf{e}_1$. 
\begin{Coro}\label{coro: pionneers} Let $d>6$ and assume that \eqref{eq: 2pt full space estimate} holds. Then, there exists $C>0$ such that, for every $n\geq 0$,
\begin{equation}
	\varphi_{p_c}(\mathbb H_n):=p_c\sum_{\substack{x\in \mathbb H_n\\y\notin \mathbb H_n\\x\sim y}}\tau_{\mathbb H_n}(0,x)\leq C.
\end{equation}
\end{Coro}
\begin{Rem} $(i$) This result was also derived in \cite{DumPan24Perco} in the context of spread-out percolation, and in \cite{DumPan24WSAW} in the context of the weakly self-avoiding walk model.

$(ii)$ Corollary \ref{coro: pionneers} implies the uniform boundedness of the expected number of critical \emph{pioneers} of half-spaces.
\end{Rem}
\begin{proof}[Proof of Corollary \textup{\ref{coro: pionneers}}] By translation invariance, one has
\begin{equation}\label{eq:proof coro 1}
	\varphi_{p_c}(\mathbb H_n)=p_c\sum_{\substack{x\in  \mathbb H_{-n}\\y\notin \mathbb H_{-n}\\x\sim y }}\tau_{\mathbb H}(0,x).
\end{equation}
By Theorem \ref{thm:main}, there exists $C_1>0$ such that, for every $x\in \mathbb H$
\begin{equation}\label{eq:proof coro 2}
	\tau_{\mathbb H}(0,x)\leq C_1\frac{1+\min(x_1,|x|)}{1+|x|^d}.
\end{equation}
Plugging \eqref{eq:proof coro 2} in \eqref{eq:proof coro 1} concludes the proof.
\end{proof}
Finally, let us mention that the upper bound in Theorem \ref{thm:main} is also useful in the recent \cite{Asselah2025capacity} (see Lemma 4.3 there).

\paragraph{Notations.} We let $\Vert \cdot \Vert$ (resp. $|\cdot |$) denote the standard Euclidean norm (resp. the $\ell^\infty$ norm) on $\mathbb R^d$. If $f,g>0$, we write $f\lesssim g$ (or $g\gtrsim f$) if there exists $C>0$, which only depends on $d$ (and potentially the spread parameter $L$) such that $f\leq Cg$. If $f\lesssim g$ and $g\lesssim f$, we write $f\asymp g$.

Given $A,B,C\subset \mathbb Z^d$, we write $\{A\connect{C\:}B\}$ for the event that there exists an open path in $C$ connecting $A$ and $B$, and we omit the superscript $C$, when $C= \mathbb Z^d$. 

We now introduce various ``half-space notations''. Observe that these notations are slightly different from the standard ones. Given a subset $A\subset \mathbb Z^d$, we define the inner boundary of $A$ in $\mathbb H$,
\begin{equation}
\partial A := \{z \in A : \exists z'\sim z \text{ with }z'\in \mathbb H \cap A^c\},
\end{equation} where we recall that $z'\sim z$ means that $z$ and $z'$ are neighbors in the graph $\mathcal G$ under consideration.  For $z\in \mathbb H$, and $r\ge 0$, we denote the box of radius $r$ centered at $z$ in $\mathbb H$ as 
\begin{equation}
B_r(z) = \{y\in \mathbb H : |y-z |\le r\},
\end{equation}
and just write $B_r$ when $z$ is the origin.

\paragraph{The van den Berg--Kesten inequality.} If $E$ and $F$ are two percolation events, we write $E\circ F$ for the event of \emph{disjoint} occurrence of $E$ and $F$, that is, the event that there exist two disjoint sets $\mathcal I$ and $\mathcal J$ of edges such that the configuration restricted to $\mathcal I$ (resp.~$\mathcal J$) is sufficient to decide that $E$ (resp.~$F$) occurs. The van den Berg--Kesten (BK) inequality (see \cite[Section~2.3]{GrimmettPercolation1999}) states that for two \emph{increasing} events (i.e.~events that are stable under the action of opening edges) $E$ and $F$, one has
\begin{equation}\label{eq:BK ineq}
\mathbb P[E\circ F]\le \mathbb P[E]\mathbb P[F].\tag{BK}
\end{equation}

\section{Proof of Theorem~\ref{thm:main}}
In the rest of the paper, we work either in the nearest-neighbour or in the spread-out setting (with spread parameter $L\geq 1$). Additionally, we assume that $d>6$ and that \eqref{eq: 2pt full space estimate} holds. The proof of Theorem \ref{thm:main} is based on Propositions \ref{prop:chatterjeehanson} and \ref{prop:chatterjeehansonsosoe} and on two new ingredients: Propositions~\ref{prop:reversedBK} and~\ref{prop.phik}. We state these results here and prove them in later sections. 

We observe the following consequence of the BK inequality: for every $x,y\in \mathbb H$, letting $n = \lfloor|x-y|/3 \rfloor$ and assuming that $n\geq 1$ (resp. $n\geq L$ in the spread-out case), one has
\begin{equation}\label{upperboundbyBK}
\tau_{\mathbb H}(x,y) \le \sum_{u\in \partial B_n(x)} \sum_{v\in \partial B_n(y)} \tau_{B_n(x)}(x,u)\cdot \tau_{\mathbb H}(u,v) \cdot \tau_{B_n(y)}(v,y). 
\end{equation}
Indeed, exploring an open self-avoiding path $\gamma$ from $x$ to $y$, and decomposing it according to the last vertex $u$ visited by $\gamma$ before exiting $B_n(x)$, and the first vertex $v$ such that the restriction of $\gamma$ to the portion between $v$ and $y$ lies in $B_n(y)$ gives
\begin{equation}
	\{x\connect{\mathbb H\:}y\}\subset \bigcup_{u\in \partial B_n(x)}\bigcup_{v\in \partial B_n(y)}\{x\connect{B_n(x)\:}u\}\circ \{u\connect{\mathbb H\:}v\}\circ \{v\connect{B_n(y)\:}y\}.
\end{equation}
See Figure \ref{fig:decomp} for an illustration.
Using a union bound and \eqref{eq:BK ineq} gives \eqref{upperboundbyBK}. 
\begin{figure}[htb]
	\begin{center}
		\includegraphics[scale=1]{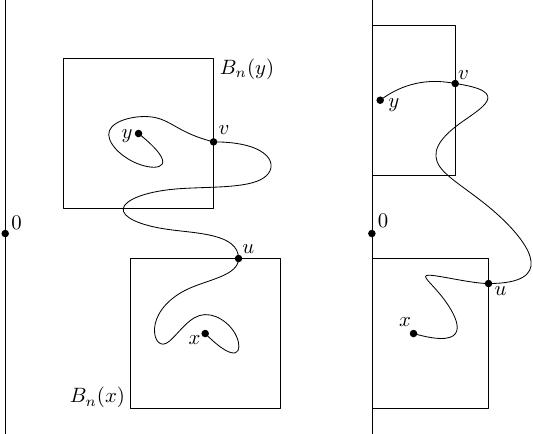}
		\caption{An illustration of the decomposition used to obtain \eqref{upperboundbyBK} (in the nearest-neighbour case). The black bold path represents an open self-avoiding path from $x$ to $y$. Depending on the values of $x_1$ and $y_1$, the boxes $B_n(x)$ and $B_n(y)$ may ``touch'' the boundary of $\mathbb H$. The reversed inequality of Proposition \ref{prop:reversedBK} decomposes paths from $x$ to $y$ similarly, except that there is an additional restriction to vertices $u$ and $v$ satisfying $u_1,v_1\geq \varepsilon n$.}
		\label{fig:decomp}
	\end{center}
\end{figure}

 Our first result provides a reversed inequality (see also Remark~\ref{rem:rem} below) up to some small multiplicative constant. Given $\varepsilon\in (0,1)$, $r\ge 0$ and $x\in \mathbb H$, we write 
\begin{equation}\label{Brepsilon}
\partial B_r^\varepsilon(x) = \{u\in \partial B_r(x) : u_1\ge \varepsilon r\}.
\end{equation}
\begin{Prop}\label{prop:reversedBK}
For every $\varepsilon \in (0,1/2)$, there exist $c,n_0>0$, such that for every $x,y \in \mathbb H$, letting $n  = \lfloor|x-y|/3\rfloor$ and assuming that $n\geq n_0$, one has 
\begin{equation}
\tau_{\mathbb H}(x,y) \ge c \sum_{ u\in \partial B_n^\varepsilon(x)} \sum_{v\in \partial B_n^\varepsilon(y)} \tau_{B_n(x)}(x,u) \cdot \tau_{\mathbb H}(u,v) \cdot \tau_{B_n(y)}(v,y).
\end{equation} 
\end{Prop} 
The second important new ingredient is the following estimate. 
\begin{Prop} \label{prop.phik}	
For every $n\geq 1$ and every $x\in \mathbb H$,  
	\begin{equation}\label{eq:pioneer1} 
	\sum_{u\in \partial B_n(x)} \tau_{B_n(x)}(x,u) \lesssim \frac {1+\min(x_1,n)} n. 
	\end{equation}
	Furthermore, there exists $\varepsilon\in (0,1/2)$, such that, for every $n\geq 1$ and every $x\in \mathbb H$, 
	\begin{equation}\label{eq:pioneer2}
	\sum_{u\in \partial B_n^\varepsilon(x)} \tau_{B_n(x)}(x,u) \gtrsim \frac {1+\min(x_1,n)} n. 
	\end{equation} 
\end{Prop} 
\begin{Rem} $(1)$ The case $x_1=0$ of \eqref{eq:pioneer1} was derived in \cite[Lemma~26]{ChatterjeeHanson}. It is also known (see \cite[Theorem~1.5]{vdHofstadSapoS14}) that for every $n\geq 1$ and every $x\in \mathbb H$, 
\begin{equation}
	\sum_{u\in \partial B_n(x)}\tau_{B_n(x)}(x,u)\lesssim 1.
\end{equation}
$(2)$ A weak version of \eqref{eq:pioneer1} was derived in \cite{DumPan24Perco} in the context of (sufficiently) spread-out percolation. There, the authors obtained (see \cite[Lemma~3.5]{DumPan24Perco}) the existence of $c_0=c_0(d)>0$ such that for every $n\geq 1$ and every $x\in \mathbb H$,
\begin{equation} 
	\sum_{u\in \partial B_n(x)} \tau_{B_n(x)}(x,u) \lesssim \Big(\frac {1+\min(x_1,n)} n\Big)^{c_0}. 
	\end{equation}
\end{Rem}
We postpone the proofs of these two propositions and give a short proof of our main result. 

\begin{proof}[Proof of Theorem~\textup{\ref{thm:main}}] We begin with the upper bound. By \eqref{eq: 2pt full space estimate}, one has, for every $u,v\in \mathbb H$,
\begin{equation}\label{2pointfull}
\tau_{\mathbb H}(u,v) \le \tau(u,v) \lesssim \frac 1{1+|u-v|^{d-2}}.
\end{equation}
Together with~\eqref{upperboundbyBK} and~\eqref{eq:pioneer1}, this yields the existence of $C_1>0$ such that, for every $x,y\in \mathbb H$ with $n=\lfloor|x-y|/3\rfloor\geq 1$ (resp. $n\geq L$ in the spread-out case), 
\begin{align}
\begin{split}
\tau_{\mathbb H}(x,y)&\lesssim \frac 1{n^{d-2}} \Big(\sum_{u\in \partial B_n(x)} \tau_{B_n(x)}(x,u)\Big) \cdot \Big(\sum_{v\in \partial B_n(y)} \tau_{B_n(y)}(y,v)\Big) \\&\lesssim \frac{(1+r_{x,y}) \cdot (1+r_{y,x})}{n^d}\leq C_1 \frac{(1+r_{x,y}) \cdot (1+r_{y,x})}{1+|x-y|^d},
\end{split}
\end{align}
where (recall that $r_{x,y}=\min(x_1,|x-y|)$) we used that $\min(x_1,n)\lesssim r_{x,y}$ and $\min(y_1,n)\lesssim r_{y,x}$.
In the situation where $n=0$ (resp. $n<L$ in the spread-out case), a similar bound holds trivially (to the cost of potentially increasing $C_1$), 
since  $\max(r_{x,y},r_{y,x})\leq \max(1,L)$.

The lower bound follows similarly by combining Proposition~\ref{prop:reversedBK} together with~\eqref{CH2} and~\eqref{eq:pioneer2}. Indeed, let $\varepsilon\in (0,1/2)$ be given by Proposition~\ref{prop.phik} and $n_0=n_0(\varepsilon)$ be given by Proposition \ref{prop:reversedBK}, and observe that for every $x,y\in \mathbb H$ with $n=\lfloor|x-y|/3\rfloor\geq n_0$,
\begin{align}\label{eq:proof main thm 1}
\begin{split}	
\tau_{\mathbb H}(x,y) & \stackrel{\phantom{\eqref{eq:pioneer2}}}\gtrsim \sum_{u\in \partial B_n^\varepsilon(x)} \sum_{v\in \partial B_n^\varepsilon(y)} \tau_{B_n(x)}(x,u) \cdot \tau_{\mathbb H}(u,v) \cdot \tau_{B_n(y)}(v,y)\\ 
& \stackrel{\phantom{\eqref{eq:pioneer2}}}\gtrsim  \min_{\substack{u\in \partial B_n^\varepsilon(x)\\v\in \partial B_n^\varepsilon(y)}}\tau_{\mathbb H}(u,v)\cdot  \Big(\sum_{u\in \partial B_n^\varepsilon(x)} \tau_{B_n(x)}(x,u)\Big) \cdot \Big(\sum_{v\in \partial B_n^\varepsilon(y)} \tau_{B_n(y)}(y,v)\Big) \\&\stackrel{\eqref{eq:pioneer2}}\gtrsim  \min_{\substack{u\in \partial B_n^\varepsilon(x)\\v\in \partial B_n^\varepsilon(y)}}\tau_{\mathbb H}(u,v)\cdot\frac{(1+r_{x,y}) \cdot (1+r_{y,x})}{n^2},
\end{split} 
\end{align}
where we used that $\min(x_1,n)\gtrsim r_{x,y}$ and $\min(y_1,n)\gtrsim r_{y,x}$. Observe that if $u\in \partial B_n^\varepsilon(x)$ and $v\in \partial B_n^\varepsilon(y)$, then $\tfrac{n}{3}\leq |u-v|\leq \frac{10}{\varepsilon}\min(u_1,v_1)$. Using \eqref{CH1} with $K=10/\varepsilon$ gives $c_1=c_1(\varepsilon)$ such that,
\begin{equation}\label{eq:proof main thm 2}
	\min_{\substack{u\in \partial B_n^\varepsilon(x)\\v\in \partial B_n^\varepsilon(y)}}\tau_{\mathbb H}(u,v)\geq \frac{c_1}{n^{d-2}}.
\end{equation}
Plugging \eqref{eq:proof main thm 2} in \eqref{eq:proof main thm 1} gives $c_2>0$ such that, for every $x,y\in \mathbb H$ such that $n\geq n_0$,
\begin{equation}
	\tau_{\mathbb H}(x,y)\gtrsim \frac{(1+r_{x,y}) \cdot (1+r_{y,x})}{n^d}\geq c_2 \frac { (1+ r_{x,y}) \cdot (1+ r_{y,x})}{1+ |x-y|^d}.
\end{equation}
Again, the bound in the case $n\leq n_0$ follows straightforwardly (to the cost of potentially decreasing $c_2$). This concludes the proof.
\end{proof}

\begin{Rem}\label{rem:rem}
Retrospectively, by combining \eqref{eq: 2pt full space estimate}, Theorem~\ref{thm:main}, and \eqref{eq:pioneer1}, we can deduce that there exists some constant $c>0$ such that, for every $x,y\in \mathbb H$ with $n=\lfloor |x-y|/3\rfloor\geq 1$, 
\begin{equation}\label{eq:improved reversed sl}
\tau_{\mathbb H}(x,y) \ge c \sum_{ u\in \partial B_n(x)} \sum_{v\in \partial B_n(y)} \tau_{B_n(x)}(x,u) \cdot \tau_{\mathbb H}(u,v) \cdot \tau_{B_n(y)}(v,y),
\end{equation}
which strengthens the result of Proposition~\ref{prop:reversedBK}. Let us provide a short proof of \eqref{eq:improved reversed sl}. For $u,v$ as above, \eqref{eq: 2pt full space estimate} gives that $\tau_{\mathbb H}(u,v)\lesssim \frac{1}{|x-y|^{d-2}}=\frac{1}{n^{d-2}}$, and \eqref{eq:pioneer1} gives that
\begin{equation}
	\sum_{u\in \partial B_n(x)}\tau_{B_n(x)}(x,u)\lesssim \frac{1+\min(x_1,n)}{n}, \qquad \sum_{v\in \partial B_n(y)}\tau_{B_n(y)}(y,v)\lesssim \frac{1+\min(y_1,n)}{n}.
\end{equation}
Combining these observations, we get that the sum on the right-hand side of \eqref{eq:improved reversed sl} is bounded by a quantity of order $\frac{(1+\min(x_1,n))(1+\min(y_1,n))}{n^d}$, which is itself bounded by (a multiple of) $\tau_{\mathbb H}(x,y)$ by Theorem \ref{thm:main}.
 
\end{Rem}

\section{Proof of Proposition~\ref{prop.phik} } 
We now turn to the proof of Proposition \ref{prop.phik}.
\begin{proof}[Proof of~\eqref{eq:pioneer1}] Let $n\geq 1$.
The case $x_1=0$ (resp. $x_1\leq L-1$ in the spread-out case) was derived\footnote{For full disclosure, \cite{ChatterjeeHanson} only treats the case $x_1=0$. However, it is easy to extend their result to the case $x_1\leq L-1$ by using the Fortuin--Kasteleyn--Ginibre (FKG) inequality (see \cite[Chapter~2.2]{GrimmettPercolation1999}).} in~\cite[Lemma~26]{ChatterjeeHanson}. We now consider a general point $x \in \mathbb H$. Decomposing an open self-avoiding path $\gamma$ from $x$ to $u$ according to the earliest point $v\in \partial B_{n/4}(u)$ (along $\gamma$) such that the portion of $\gamma$ between $v$ and $u$ lies in $B_{n/4}(u)$ and using \eqref{eq:BK ineq}, we obtain
for any $u\in \partial B_n(x)$,
\begin{equation}\label{eq:proof prop 2.2 1}
\mathbb P[x \connect{B_n(x)\:} u] \le \sum_{v \in \partial B_{n/4}(u) \cap B_n(x)} \mathbb P[x \connect{B_n(x)\:} v] \cdot \mathbb P[  v \connect{B_{n/4}(u) \cap B_n(x)\:} u].
\end{equation}
Write $x=(x_1,x_\bot)$, with $x_\bot\in \mathbb Z^{d-1}$. For every $v=(v_1,v_\bot)\in \partial B_{n/4}(u)\cap B_n(x)$, letting $\tilde{x}=x-(0,v_\bot)$ and $\tilde{v}=v-(0,x_\bot)$, \eqref{CHS1} gives that
\begin{equation}\label{eq:proof prop 2.2 2}  
	\mathbb P[x\connect{B_n(x)\:}v]=\mathbb P[\tilde{x}\connect{B_n(\tilde{x})\:}v_1\mathbf{e}_1]=\mathbb P[\tilde{v}\connect{B_n(x_1\mathbf{e}_1)\:}x_1\mathbf{e}_1]\lesssim \frac{\min(1+v_1,1+x_1)}{1+|x-v|^{d-1}}\lesssim \frac{1+\min(x_1,n)}{n^{d-1}},
\end{equation}
where we used translation invariance in the first two equalities. Plugging \eqref{eq:proof prop 2.2 2} in \eqref{eq:proof prop 2.2 1} and using  \eqref{eq:pioneer1} for $x_1=0$ (resp. $x_1\leq L-1$), we deduce that, for every $x\in \mathbb H$ and $u\in \partial B_n(x)$,
\begin{equation}  
\mathbb P[x \connect{B_n(x)\:} u]  \lesssim  \frac{ 1+\min(x_1,n)}{n^{d-1}} \cdot \Big(\sum_{v \in \partial B_{n/4}(u) \cap B_n(x)}  \mathbb P[  v \connect{B_{n/4}(u) \cap B_n(x)\:} u]\Big)  \lesssim \frac {1+\min(x_1,n)}{n^d}.
\end{equation}
Summing over $u \in \partial B_n(x)$ concludes the proof. 
\end{proof} 

\begin{proof}[Proof of~\eqref{eq:pioneer2}] Let $\varepsilon\in (0,1/2)$ to be fixed. Let $n \geq 1$ and $x\in \mathbb H$. Without loss of generality, we may assume that $x=(x_1,0,\ldots,0)$. We first assume that $n\geq \tfrac{1}{2}x_1$. Recall that $\mathbf{e}_1=(1,0,\ldots,0)$. On the one hand, by\footnote{Let us give more details. If $\tfrac{1}{2}x_1\leq n \leq 2x_1$, then \eqref{CH1} (for a proper choice of $K$) gives \eqref{ukun}. If $n\geq 2x_1$, we may apply \eqref{CHS2} (since $x_1+2n\geq 4x_1$) to get \eqref{ukun}.} \eqref{CH1} and \eqref{CHS2} , one has  
\begin{equation}\label{ukun}
\tau_{\mathbb H}(x,x+2n\mathbf{e}_1) \gtrsim \frac{ 1+\min(x_1,n)}{n^{d-1}}. 
\end{equation}
On the other hand, decomposing an open self-avoiding path from $x$ to $x+2n\mathbf{e}_1$ according to the first point in $\partial B_n(x)$ it visits and using \eqref{eq:BK ineq} gives
\begin{multline}\label{eq:proof pionner2 1}
\tau_{\mathbb H}(x,x+2n\mathbf{e}_1)  \lesssim \max_{u\in \partial B_n^\varepsilon(x)}\tau_{\mathbb H}(u,x+2n\mathbf{e}_1) \cdot \Big(\sum_{u\in \partial B_n^\varepsilon(x)} \tau_{B_n(x)}(x,u)\Big) \\+ \max_{u\in \partial B_n(x)\setminus \partial B_n^\varepsilon(x)}\tau_{\mathbb H}(u,x+2n\mathbf{e}_1) \cdot \Big( \sum_{u \in \partial B_n(x) \setminus \partial B_n^\varepsilon(x)}\tau_{B_n(x)}(x,u)\Big).
\end{multline} 
Using \eqref{eq: 2pt full space estimate} and \eqref{CHS1} give that, for $n$ large enough (in terms of $\varepsilon$),
\begin{equation}\label{eq:proof pionner2 2}
	\max_{u\in \partial B_n^\varepsilon(x)}\tau_{\mathbb H}(u,x+2n\mathbf{e}_1)\lesssim \frac{1}{n^{d-2}}, \qquad \max_{u\in \partial B_n(x)\setminus \partial B_n^\varepsilon(x)}\tau_{\mathbb H}(u,x+2n\mathbf{e}_1)\lesssim \frac{\varepsilon n}{n^{d-1}}.
\end{equation}
Combining \eqref{ukun}, \eqref{eq:proof pionner2 1}, and \eqref{eq:proof pionner2 2} gives, for $n$ large enough,
\begin{equation}
	\frac{ 1+\min(x_1,n)}{n^{d-1}}\lesssim \frac{1}{n^{d-2}}\cdot \Big(\sum_{u\in \partial B_n^\varepsilon(x)} \tau_{B_n(x)}(x,u)\Big)+\varepsilon\cdot \frac{ 1+\min(x_1,n)}{n^{d-1}},
\end{equation}
where we used \eqref{eq:pioneer1} to get $\sum_{u \in \partial B_n(x) \setminus \partial B_n^\varepsilon(x)}\tau_{B_n(x)}(x,u)\lesssim \frac{ 1+\min(x_1,n)}{n^{}} $. Choosing $\varepsilon\in(0,1/2)$ small enough concludes the proof in the case $n\geq \tfrac{1}{2}x_1$ and $n$ large enough (in terms of $\varepsilon$). The remaining values of $n$ can be handled by adjusting the value of the constant in \eqref{eq:pioneer2}.

It remains to treat the case $n<\tfrac{1}{2}x_1$. Observe that for this choice, one has $B_n(x)=\{u\in \mathbb Z^d : |u-x|\leq n\}$ (that is, the $\mathbb Z^d$-box of radius $n$ around $x$ is fully included in $\mathbb H$). Thus, we may use \cite{DuminilTassionNewProofSharpness2016} (which gives that $\varphi_{p_c}(B_n(x)):=\sum_{\substack{u\in \partial B_n(x)\\v\notin B_n(x)\\u\sim v}}\tau_{B_n(x)}(x,u)p_c\geq 1$) to conclude that, for any $\varepsilon\in (0,1/2)$,
\begin{equation}
	\sum_{u\in \partial B_n^\varepsilon(x)}\tau_{B_n(x)}(x,u)\geq \frac{1}{2}\sum_{u\in \partial B_n(x)}\tau_{B_n(x)}(x,u)\gtrsim 1.
\end{equation} 
This concludes the proof. 
\end{proof}

\section{Proof of Proposition~\ref{prop:reversedBK}}
The proof of Proposition \ref{prop:reversedBK} is technically more involved. We will need a number of preliminary results and notations.

We rely on the notion of \emph{regular points} initially introduced in \cite{KozmaNachmias}, and recently revisited in \cite{Asselah2025capacity}. In particular, we take advantage of the geometric and convenient definition of regular points chosen in \cite{Asselah2025capacity}. Most of the arguments below are adaptations of intermediate results that already appeared in this paper. It is worth mentioning that our setting is often simpler. For instance, our Lemma \ref{Prop:regpoints} requires an averaged estimate on the expected number of regular points, while \cite[Section~6]{Asselah2025capacity} derives finer and stronger estimates (quite similarly to what is done in \cite{KozmaNachmias}). Our main contribution is to use these regular points to derive ``reversed Simon--Lieb type inequalities''.

 In Section~\ref{sec:lowerbound2sets}, we present a general lower bound for the probability that two sets are connected, which we express in terms of their $(d-4)$-capacity. In Section~\ref{sec:regularity}, we recall all the necessary definitions to introduce the theory of regularity of~\cite{KozmaNachmias}, taking here the viewpoint of~\cite{Asselah2025capacity}. Finally, we give the proof of Proposition \ref{prop:reversedBK} in Section~\ref{sec:conclusion}. Recall that we have assumed that $d>6$ and that \eqref{eq: 2pt full space estimate} holds.

\subsection{A general lower bound for the connection probability of two finite sets}\label{sec:lowerbound2sets}
We present here a general lower bound for the probability that two finite subsets of $\mathbb H$ are connected by an open path in terms of the product of their $(d-4)$-capacity. A non-restricted version of this result has already appeared in~\cite{Asselah2025capacity}. We adapt their argument to the setting of restricted percolation on $\mathbb H$. Since the proof in~\cite{Asselah2025capacity} is only sketched, we provide here a more detailed argument for the reader's convenience. 

Recall that $\Vert \cdot \Vert$ denotes the Euclidean norm on $\mathbb R^d$. Given a non-empty finite subset $A\subset \mathbb Z^d$, we define its $(d-4)$-capacity as 
\begin{equation}
\textrm{Cap}_{d-4}(A) := \Big( \inf\Big\{ \sum_{a,b\in A} \mu(a)\mu(b) (1+\|a-b\|)^{4-d} : \mu \textrm{ probability measure on }A\Big\}\Big)^{-1}.
\end{equation} 
Given two finite sets $A,B\subset \mathbb Z^d$, we let ${\textup{d}}(A,B) := \min_{a\in A,b\in B} \|a-b\|$, and $\textrm{diam}(A) := \max_{a,a'\in A} \|a-a'\|$.  

\begin{Lem}\label{lem:2sets}
For every $c_1>0$, there exists $c>0$ such that the following holds. For every finite $A, B\subseteq \mathbb H$ such that $\textup{d}(A,B) \ge c_1\cdot \max(\textup{diam}(A),\textup{diam}(B))$, one has
\begin{equation}
\mathbb P[A\connect{\mathbb H\:} B] \geq c \cdot \frac{\big(\min_{a\in A,b\in B} \tau_{\mathbb H}(a,b)\big)^2}{d(A,B)^{2-d}} \cdot \textup{Cap}_{d-4}(A)\cdot \textup{Cap}_{d-4}(B).
\end{equation}
\end{Lem}
\begin{proof} 
As for Lemma 8.1 in~\cite{Asselah2025capacity}, the proof is based on a second moment method. More precisely, given two probability measures $\mu$ and $\nu$ supported respectively on $A$ and $B$, consider the random variable
\begin{equation}
	X = \sum_{a\in A} \sum_{b\in B} \mu(a)\nu(b)\cdot \mathbf 1\{a\connect{\mathbb H\:} b\}.  
\end{equation}
Using Cauchy--Schwarz's inequality, we get 
\begin{equation}\label{CSineq}
\mathbb P[A\connect{\mathbb H\:} B] \ge \mathbb P[X>0] \ge \frac{\mathbb E[X]^2}{\mathbb E[X^2]}.
\end{equation}
The first moment $\mathbb E[X]$ can be easily lower bounded: using that $\mu$ and $\nu$ are probability measures, we obtain that 
\begin{equation}\label{firstmoment}
\mathbb E[X] \ge \min_{a\in A,b\in B} \tau_{\mathbb H}(a,b).
\end{equation}
We now upper bound $\mathbb E[X^2]$. Write
\begin{equation}
\mathbb E[X^2]  =\sum_{\substack{a,a'\in A\\ b,b'\in B}} \mu(a)\mu(a')\nu(b)\nu(b')\mathbb P[a\connect{\mathbb H\:} b, a'\connect{\mathbb H\:} b'].
\end{equation}
Exploring an open self-avoiding path from $a$ to $b$, and then one from $a'$ to $b'$, and using \eqref{eq:BK ineq}, we obtain that for every $a,a'\in A$ and $b,b'\in B$, 
\begin{multline}\label{eq.aa'bb'}
 \mathbb P[a\connect{\mathbb H\:}b, a'\connect{\mathbb H\:}b']   \le \tau_{\mathbb H}(a,b)  \tau_{\mathbb H}(a',b') 
   + \sum_{w,w'\in \mathbb H} \tau_{\mathbb H}(a,w)\tau_{\mathbb H}(a',w)  \tau_{\mathbb H}(w,w')\tau_{\mathbb H}(w',b)\tau_{\mathbb H}(w',b') 
  \\+ \sum_{w,w'\in \mathbb H}\tau_{\mathbb H}(a,w)\tau_{\mathbb H}(w,b')\tau_{\mathbb H}(w,w') \tau_{\mathbb H}(w',a')\tau_{\mathbb H}(w',b). 
 \end{multline} 
 See Figure \ref{fig:bk} for an illustration.
 \begin{figure}[htb]
 	\begin{center}
 		\includegraphics{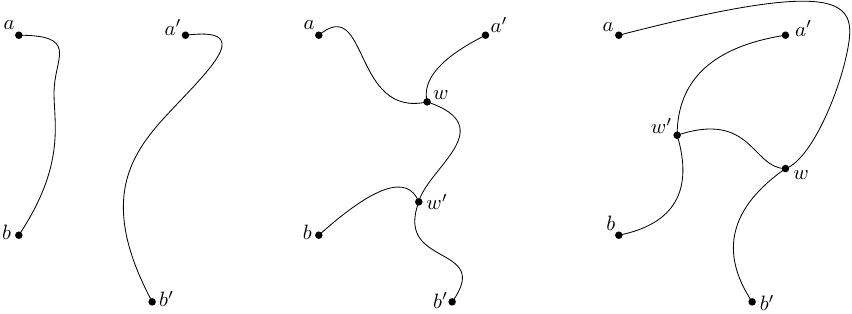}
 		\caption{An illustration of the diagrams underlying the proof of \eqref{eq.aa'bb'}. The black bold paths are open self-avoiding paths. If the event $\{a\connect{\mathbb H\:}b, a'\connect{\mathbb H\:}b'\}$ occurs, then one of the situations must occur (for some $w,w'\in \mathbb H$). Each diagram corresponds to a term on the right-hand side of \eqref{eq.aa'bb'}.}
 		\label{fig:bk}
 	\end{center}
 \end{figure}
 
 We begin with the analysis of the first sum on the right-hand side of \eqref{eq.aa'bb'}. Letting $r= \textup{d}(A,B)/2$ and using \eqref{eq: 2pt full space estimate}, we find that 
 \begin{align}\label{eq.aa'bb' bis}
 	\begin{split}
\Sigma(a,a',b,b')& := \sum_{w,w'\in \mathbb H} \tau_{\mathbb H}(a,w)\tau_{\mathbb H}(a',w)  \tau_{\mathbb H}(w,w')\tau_{\mathbb H}(w',b)\tau_{\mathbb H}(w',b')  \\
 &\lesssim \frac 1{r^{d-2}} \Big(\sum_{w\in \mathbb H} \tau_{\mathbb H}(a,w)\tau_{\mathbb H}(a',w)\Big)\cdot \Big(\sum_{w'\in \mathbb H} \tau_{\mathbb H}(w',b)\tau_{\mathbb H}(w',b')\Big)  \\
 & \qquad+  \sum_{\substack{w,w'\in \mathbb H\\\| w-w'\|\le r}}  \tau_{\mathbb H}(a,w)\tau_{\mathbb H}(a',w)  \tau_{\mathbb H}(w,w')\tau_{\mathbb H}(w',b)\tau_{\mathbb H}(w',b'). 
 \end{split}
 \end{align} 
Using \eqref{eq: 2pt full space estimate} and a classical estimate (see for instance \cite[Proposition~B.1]{DumPan24Perco}), 
\begin{equation}\label{eq:bounda}
\sum_{w\in \mathbb H} \tau_{\mathbb H}(a,w)\tau_{\mathbb H}(a',w) \le \sum_{w\in \mathbb Z^d} \tau(a,w) \tau(w,a') \lesssim (1+ \|a-a'\|)^{4-d}, 
\end{equation}
and likewise, 
\begin{equation}\label{eq:boundb} 
\sum_{w'\in \mathbb H} \tau_{\mathbb H}(w',b)\tau_{\mathbb H}(w',b') \lesssim (1+ \|b-b'\|)^{4-d}.
\end{equation} 
We now look at the second sum on the second line of \eqref{eq.aa'bb' bis}. If $w$ is at distance at least $r/2$ from $A$, then,
\begin{multline}
	\sum_{\substack{w'\\w: \textup{d}(w,A)\geq r/2, \: \Vert w-w'\Vert \leq r}}\tau_{\mathbb H}(a,w)\tau_{\mathbb H}(a',w)  \tau_{\mathbb H}(w,w')\tau_{\mathbb H}(w',b)\tau_{\mathbb H}(w',b')\\\lesssim \frac{1}{r^{2d-4}}\sum_{w'}\tau_{\mathbb H}(w',b)\tau_{\mathbb H}(w',b')\sum_{w: \Vert w-w'\Vert \leq r}\tau(w,w')\lesssim \frac{1}{r^{2d-6}}(1+\Vert b-b'\Vert)^{4-d},
\end{multline}
where in the first inequality we used \eqref{eq: 2pt full space estimate}, and in the second one we used \eqref{eq:boundb} and \eqref{eq: 2pt full space estimate} one more time to get $\sum_{w: \Vert w-w'\Vert \leq r}\tau(w,w')\lesssim r^2$. Similarly,
\begin{equation}
	\sum_{\substack{w\\w': \textup{d}(w',B)\geq r/2, \: \Vert w-w'\Vert \leq r}}\tau_{\mathbb H}(a,w)\tau_{\mathbb H}(a',w)  \tau_{\mathbb H}(w,w')\tau_{\mathbb H}(w',b)\tau_{\mathbb H}(w',b')\lesssim \frac{1}{r^{2d-6}}(1+\Vert a-a'\Vert)^{4-d}.
\end{equation}
By definition of $r$, it is impossible to have $\textup{d}(w,A)<r/2$, $\textup{d}(w',B)<r/2$, and $\Vert w-w'\Vert$ simultaneously. Hence,
\begin{multline}
	\sum_{\substack{w,w'\in \mathbb H\\\| w-w'\|\le r}}  \tau_{\mathbb H}(a,w)\tau_{\mathbb H}(a',w)  \tau_{\mathbb H}(w,w')\tau_{\mathbb H}(w',b)\tau_{\mathbb H}(w',b')\\\lesssim \frac 1{r^{2d-6}} \Big((1+ \|a-a'\|)^{4-d} + (1+ \|b-b'\|)^{4-d}\Big).
\end{multline}
Therefore, if we define for a probability measure $\rho$, its energy as 
\begin{equation}
\mathcal E_{d-4}(\rho) :=  \sum_{u,v\in \mathbb Z^d} \rho(u)\rho(v) (1+\|u-v\|)^{4-d},
\end{equation} 
one obtains from the previously displayed equations that
\begin{align}\label{sumSigma} 
\begin{split}
\sum_{\substack{a,a'\in A\\b,b'\in B}} \mu(a)\mu(a') \nu(b)\nu(b') \Sigma(a,a',b,b') & \lesssim \frac 1{r^{d-2}} \mathcal E_{d-4}(\mu) \mathcal E_{d-4}(\nu) + \frac 1{r^{2d-6}} \Big(\mathcal E_{d-4}(\mu) + \mathcal E_{d-4}(\nu)\Big)\\
&\lesssim  \frac 1{r^{d-2}} \cdot \mathcal E_{d-4}(\mu) \mathcal E_{d-4}(\nu), 
\end{split}
\end{align} 
where in the second inequality, we used that, since (by hypothesis) $r\gtrsim  \max(\textrm{diam}(A), \textrm{diam}(B))$, one has $\mathcal E_{d-4}(\rho) \gtrsim \frac{1}{r^{d-4}}$, for any probability measure $\rho$ supported on $A$ or $B$.  

Similar computations allow to treat  the last sum in the right-hand side of~\eqref{eq.aa'bb'}. More precisely, using repeatedly \eqref{eq: 2pt full space estimate}, we can see that it is upper bounded up to some multiplicative constant by 
\begin{align}\label{eq:second sum second moment}
\begin{split}
& \frac{1}{r^{d-2}} \Big(\sum_{w\in \mathbb H}  \tau_{\mathbb H}(a,w)\tau_{\mathbb H}(w,b')\Big) \cdot \Big( \sum_{w'\in \mathbb H} \tau_{\mathbb H}(w',a')\tau_{\mathbb H}(w',b)\Big) \\
& \qquad\qquad+ \sum_{\substack{w,w'\in \mathbb H\\\|w-w'\|\le r}} \tau_{\mathbb H}(a,w)\tau_{\mathbb H}(w,b')\tau_{\mathbb H}(w,w') \tau_{\mathbb H}(w',a')\tau_{\mathbb H}(w',b) \\
&\lesssim \frac 1{r^{d-2}}\Big\{\frac 1{r^{2(d-4)}} +  \frac 1{r^{d-2}} \sum_{\substack{w,w'\in \mathbb H\\\|w-w'\|\le r}} \Big( \tau_{\mathbb H}(a,w)\tau_{\mathbb H}(w,w') \tau_{\mathbb H}(w',a') + \tau_{\mathbb H}(w,b')\tau_{\mathbb H}(w,w') \tau_{\mathbb H}(w',b)\Big)\Big\} \\
& \lesssim \frac 1{r^{d-2}}\Big\{\mathcal E_{d-4}(\mu)\mathcal E_{d-4}(\nu) + \frac 1{r^{d-2}} \Big(\frac 1{1+\|a-a'\|^{d-6}} + \frac 1{1+\|b-b'\|^{d-6}}\Big) \Big\} \\
& \lesssim  \frac 1{r^{d-2}}\Big\{\mathcal E_{d-4}(\mu)\mathcal E_{d-4}(\nu) + \frac {\mathcal E_{d-4}(\nu)}{1+\|a-a'\|^{d-4}} + \frac {\mathcal E_{d-4}(\mu)}{1+\|b-b'\|^{d-4}} \Big\}, 
\end{split}
\end{align} 
where in the second inequality, we used \eqref{eq: 2pt full space estimate} to argue that
\begin{equation}
	\sum_{\substack{w,w'\in \mathbb H\\\|w-w'\|\le r}}\tau_{\mathbb H}(a,w)\tau_{\mathbb H}(w,w') \tau_{\mathbb H}(w',a')\leq \sum_{w,w'\in \mathbb Z^d}\tau(a,w)\tau(w,w') \tau(w',a') \lesssim \frac{1}{1+\Vert a-a'\Vert^{d-6}}.
\end{equation}
Summing \eqref{eq:second sum second moment} over $a,a'\in A$, and $b,b'\in B$ against $\mu(a)\mu(a')\nu(b)\nu(b')$, we obtain the same upper bound as in~\eqref{sumSigma}, and conclude that
\begin{equation}
\mathbb E[X^2] \lesssim \frac 1{r^{d-2}} \cdot \mathcal E_{d-4}(\mu) \mathcal E_{d-4}(\nu).
\end{equation}
Optimising over the choices of $\mu$ and $\nu$ and combining the result with~\eqref{CSineq} and~\eqref{firstmoment} concludes the proof. 
\end{proof}

 \subsection{Regular points, line good points, extended cluster} \label{sec:regularpoints}\label{sec:regularity}
We present here the basis of a technique first introduced in~\cite{KozmaNachmias} to derive the one-arm exponent in high-dimensional critical percolation. It is based on a notion of regularity. 
Here, we will adapt to our setting the definition of regular points from~\cite{Asselah2025capacity} which is stated in purely geometric terms. 

We write $\mathcal C(x;A) := \{z \in A : x \stackrel{A} \longleftrightarrow z\}$ for the cluster of a point $x$ restricted to a set $A$. 
Fix $n\ge 1$ and $x\in \mathbb H$, and for $z\in \partial B_n(x)$, and $s>0$, consider the event   
			\begin{multline}
				\mathcal T_s(z) :=  \{ |\mathcal C(z;B_n(x))\cap B_s(z)| \leq s^4 (\log s)^7\} \\\cap \{|\mathcal C(z;B_n(x))\cap B_s(z)\cap \partial B_n(x)|\leq s^2 (\log s)^7  \}. 
			\end{multline}

\begin{Def}[$K$-regular points] Given $K>0$, we call $z\in \partial B_n(x)$ a $K$-regular point, if the events $\mathcal T_s(z)$ hold for all $s\ge K$. 
\end{Def}
Let $\varepsilon \in [0,1/2)$, We denote by $X_n^{\varepsilon,K-\textup{reg}}(x)$ the number of points on $\partial B_n^\varepsilon(x)$ (recall~\eqref{Brepsilon}), which are $K$-regular and connected to $x$ in $B_n(x)$. Also, denote by $X_n^{\varepsilon}(x)$ the number of points on $\partial B_n^\varepsilon(x)$ which are connected to $x$ in $B_n(x)$ (the so-called pioneers). It turns out that most of the pioneers are regular, and consequently one can show the following lemma. 
\begin{Lem}\label{Prop:regpoints}  There exist $K_0\ge 1$ and $n_0\ge 1$ such that the following holds. For every $K\geq K_0$, every $n\ge n_0$, every $\varepsilon\in [0,1/2)$, and every $x\in \mathbb H$,
\begin{equation} 
\mathbb E[X_{n}^{\varepsilon,K-\textup{reg}}(x)] \ge \frac 12 \cdot \mathbb E[X_n^\varepsilon(x)].
\end{equation}
\end{Lem}
We defer the proof of this lemma to Section~\ref{sec:proofPropreg} and introduce now the notion of $K$-line good points.
For this, one first needs to consider a maximal subset of the set of $K$-regular points of $\partial B_n^{\varepsilon}(x)$ which has the property that all its points are  at distance at least $2K$ one from each other. Denote by $\mathcal X_n^{\varepsilon,K-\textrm{reg}}(x)$ one such maximal subset chosen uniformly at random. Then, if $z\in\mathcal X_n^{\varepsilon,K-\textrm{reg}}(x)$, we consider a line segment of length $K$ emanating from $z$, outside $B_n(x)$ and orthogonal to its boundary (choose one arbitrarily if there are many). Call $z'$ the endpoint of this line segment. We say that $z'$ is a \textbf{$K$-line good point} if all the edges on the line segment between $z$ and $z'$ are open. More generally, for any $z\in \mathcal X_n^{\varepsilon,K-\textrm{reg}}(x)$, we denote by $L_z$ the maximal open segment emanating from $z$ orthogonally to $B_n(x)$, of length at most $K$.

We next define the \textbf{extended cluster} of $x$ in $B_n(x)$, which we denote by $\mathcal C_n^e(x)$, as the cluster of $x$ in $B_n(x)$ together with all the line segments $L_z$ for $z\in \mathcal X_n^{\varepsilon,K-\textrm{reg}}(x)$.

We say that a set $A$ is $K$-\textbf{admissible} for the pair $(x,n)$, if $\mathbb P[\mathcal C_n^e(x) = A ]>0$, and for such admissible set   
we denote by $\partial_*A$ its set of points which are at distance exactly $K$ from $B_n(x)$. Hence 
by definition $\partial_*\mathcal C_n^e(x)$ is the set of $K$-line good points.

One interest of the notion of regularity, which has been noticed and used extensively in~\cite{Asselah2025capacity}, is that in any dimension $d>6$, admissible sets have a $(d-4)$-capacity which is comparable to their cardinality. 
Indeed, the following lemma was observed in~\cite[Claim~6.1]{Asselah2025capacity}. 
\begin{Lem}\label{lem:lowercap}
There exists a constant $c>0$, such that for every $n,K\geq 1$, every $\varepsilon\in [0,1/2)$, every $x\in \mathbb H$, and every $K$-admissible set $A$ for the pair $(x,n)$, 
one has 
\begin{equation}
\textup{Cap}_{d-4}(\partial_*A) \ge c |\partial_*A|. 
\end{equation} 
\end{Lem}
\begin{proof}
For the reader's convenience, we include a short proof. By taking $\mu$ to be the uniform measure on $\partial_*A$ in the definition of the $(d-4)$-capacity, one gets 
\begin{equation}
\textrm{Cap}_{d-4}(\partial_*A) \ge \frac{ |\partial_*A|^2}{\sum_{a,a'\in \partial_*A} (1+\|a-a'\|)^{4-d}}. 
\end{equation}

Let $\omega$ be a percolation configuration realizing $\{\mathcal C_n^e(x)=A\}$ (note that it exists since $A$ is $K$-admissible for the pair $(x,n)$). By definition, to every fixed $a\in \partial_*A$ corresponds a unique $z_a\in \partial B_n(x)$ such that $\omega$ realises the events $\{z_a\connect{B_n(x)\:}x\}$ and $\mathcal T_{s}(z_a)$ for every $s\geq K$. Thus, one has that, for every $a\in \partial_*A$, and every $s\geq K$,
\begin{equation}\label{eq:bound partial A cap Bs}
	|\partial_*A\cap B_s(a)|\lesssim |\partial_*A\cap B_s(z_a)|\lesssim |\mathcal C(z_a;B_n(x))\cap B_s(z_a)\cap \partial B_n(x)|\leq s^2(\log s)^7,
\end{equation}
where the implicit constants do not depend on $x$, $n$, $K$, $\varepsilon$, and $A$. As a consequence, for every fixed $a\in \partial_*A$, we find that
\begin{equation} 
\sum_{a'\in \partial_*A} (1+\|a-a'\|)^{4-d} \lesssim 1+ \sum_{i\geq \log_2(K)}\frac{|\partial_*A\cap (B_{2^{i+1}}(a)\setminus B_{2^i}(a))|}{2^{i(d-4)}}\lesssim 1+\sum_{i \geq \log_2(K)}\frac{2^{2i}(\log 2^i)^7}{2^{i(d-4)}} \lesssim  1,
\end{equation}
where in the first inequality we used that (by definition) $(\partial_*A\setminus \{a\}) \cap B_{K}(a)=\emptyset$, in the second inequality we used \eqref{eq:bound partial A cap Bs}, and where (again) the implicit constants do not depend on $x$, $n$, $K$, $\varepsilon$, and $A$. This concludes the proof. 
\end{proof}   

\subsection{Conclusion}\label{sec:conclusion}
We now have all the necessary material to prove our desired result. 

\begin{proof}[Proof of Proposition~\textup{\ref{prop:reversedBK}}] 
Let $\varepsilon\in (0,1/2)$. Let $x,y\in \mathbb H$ and set $n = \lfloor |x- y|/3\rfloor$. Fix $K_0,n_0$ as in Lemma~\ref{Prop:regpoints}. Let $K\geq K_0$ to be chosen large enough and assume that $n\ge n_0$.  

We first observe that 
\begin{equation}
	\tau_{\mathbb H}(x,y) \ge \sum_{A,B} \mathbb P\big[\mathcal C_n^e(x) = A, \mathcal C_n^e(y) = B, \partial_*A\connect{\textrm{off } (A\cup B)\:} \partial_*B\big],
\end{equation}
where $\{\partial_*A\connect{\textrm{off } (A\cup B)\:}\partial_*B\}$ is the event that $\partial_*A$ is connected to $\partial_*B$ by an open path in $\mathbb H$ that avoids $A\cup B$, except at its end points. Note next that the two events  
$\{\mathcal C_n^e(x) = A,\mathcal C_n^e(y) = B\}$ and $\{\partial_*A\connect{\textrm{off } (A\cup B)\:} \partial_*B\}$ depend on different sets of edges, and are thus independent. Hence, we get 
\begin{equation}  
\tau_{\mathbb H}(x,y) \ge \sum_{A,B} \mathbb P\big[\mathcal C_n^e(x) = A, \mathcal C_n^e(y) = B\big] \cdot \mathbb P\big[\partial_*A\connect{\textrm{off } (A\cup B)\:} \partial_*B\big]. 
\end{equation} 
Given $A,B$ which are $K$-admissible respectively for $(x,n)$ and $(y,n)$, we define 
\begin{equation} 
C= A\cup B \cup \big(\bigcup_{a\in \partial_*A} B_K(a)\big) \cup \big(\bigcup_{b\in \partial_*B} B_K(b)\big).
\end{equation}
Now, we fix an arbitrary ordering of the elements of $\partial_*A$ and $\partial_*B$. On the event $\{\partial_*A\connect{\textrm{off } (A\cup B)\:}\partial_*B\}$, we denote by $Y_1$ the first element $a\in \partial_*A$ for this ordering such that $\partial B_K(a)$ is connected to $\cup_{b\in \partial_*B} B_K(b)$ by an open path that avoids $C$, and let $Y_2$ be the first element $b\in \partial_*B$ such that $\partial B_K(Y_1)$ is connected to $\partial B_K(b)$ by an open path that avoids $C$. Finally, we let
\begin{equation} 
H = (A \cup B) \setminus \big(\bigcup_{z \in \partial_*A \cup \partial_*B} L_z\big),
\end{equation}
where for any $a\in \partial_*A$ we let $L_a$ be the line segment of length $K$ between $a$ and $B_n(x)$, and similarly for $b\in \partial_*B$. Letting $c(K)>0$ be the probability that all edges are open in a box of size $K$, we find  
\begin{align}
\begin{split}
 \mathbb P[\partial_*A\connect{\textrm{off } (A\cup B)\:}\partial_*B\big]& \ge \sum_{a\in \partial_*A}\sum_{b\in \partial_*B} \mathbb P\big[Y_1 = a, Y_2= b, \textrm{all edges in }B_K(a)\cup B_K(b) \textrm{ are open}\big] \\
 &  =  c(K)^2 \cdot \mathbb P \Big[ \big(\cup_{a\in \partial_*A} B_K(a)\big) \connect{\textrm{off } C\:}\big(\cup_{b\in \partial_*B} B_K(b)\big)\Big]\\
 & = c(K)^2 \cdot \mathbb P \Big[ \big(\cup_{a\in \partial_*A} B_K(a)\big) \connect{\textrm{off } H\:}\big(\cup_{b\in \partial_*B} B_K(b)\big)\Big]\\
 & \ge c(K)^2 \cdot \mathbb P \big[  \partial_*A  \connect{\textrm{off } H\:}\partial_*B\big].
 \end{split} 
 \end{align}
We then write 
\begin{equation}
\mathbb P\big[  \partial_*A  \connect{\textrm{off } H\:}\partial_*B\big] = 
\mathbb P\big[  \partial_*A  \connect{} \partial_*B\big] - \mathbb P\big[  \partial_*A  \connect{\textrm{via } H\:}\partial_*B\big], 
\end{equation}
where $\{\partial_*A  \connect{\textrm{via } H\:}\partial_*B\}$ denotes the event that $\partial_*A$ and  $\partial_*B$ are connected by an open path, and all open paths that connect them intersect $H$.  We claim that for any constant $\delta>0$, one can find $K\geq K_0$ large enough, so that for all admissible sets $A$ and $B$, 
\begin{equation} \label{connect.viaH}
\mathbb P\big[  \partial_*A  \connect{\textrm{via } H\:}\partial_*B\big] \le \frac{\delta}{n^{d-2}} \cdot |\partial_*A| \cdot |\partial_*B|. 
\end{equation} 
To see this, we note that by a union bound, it suffices to show that for $K$ sufficiently large, for every $a\in \partial_*A$ and $b\in \partial_*B$, one has\begin{equation} \label{connect.viaH2}
\mathbb P\big[ a \connect{\textrm{via } H\:} b\big] \le \frac{\delta}{n^{d-2}}. 
\end{equation} 
Now, decomposing an open self-avoiding path connecting $a$ and $b$ through $H$ according to the first point in $H$ it visits, and using \eqref{eq:BK ineq} and~\eqref{eq: 2pt full space estimate}, we deduce that  
\begin{align}  \label{eq: bound with sqrt K}
\begin{split}
\mathbb P\big[ a \connect{\textrm{via } H\:}b\big]& \le \sum_{u\in H} \tau_{\mathbb H}(a,u) \tau_{\mathbb H}(u,b) \lesssim \frac 1{n^{d-2}} \Big( \sum_{u\in H\cap A} \tau_{\mathbb H}(a,u) +  \sum_{u\in H\cap B} \tau_{\mathbb H}(u,b)\Big)\\
& \lesssim \frac 1{n^{d-2}}   \sum_{i \geq \log_2(K)}\frac{2^{4i} (\log 2^i)^7}{2^{i(d-2)}} \lesssim \frac 1{n^{d-2}  \sqrt K}, 
\end{split} 
\end{align}
where we used that
\begin{equation}
	\sum_{u\in H\cap A}\tau_{\mathbb H}(a,u)\lesssim \sum_{i\geq \log_2(K)}\frac{|H\cap A\cap (B_{2^{i+1}}(a)\setminus B_{2^i-1}(a))|}{2^{i(d-2)}}\lesssim \sum_{i \geq \log_2(K)}\frac{2^{4i} (\log 2^i)^7}{2^{i(d-2)}},
\end{equation}
which follows by using that $H\cap B_{K-1}(a)=\emptyset$, and by proving---similarly to \eqref{eq:bound partial A cap Bs}---that $|H\cap A\cap B_s(a)|\lesssim s^4(\log s)^7$ for every $s\geq K$. 
As a consequence, \eqref{eq: bound with sqrt K} gives \eqref{connect.viaH2}---and hence~\eqref{connect.viaH}---by choosing $K$ large enough. 

On the other hand, by combining Lemmas~\ref{lem:2sets} and~\ref{lem:lowercap}, we obtain that for every $K$-admissible sets $A$ and $B$,
\begin{equation} 
\mathbb P\big[  \partial_*A  \longleftrightarrow \partial_*B\big]\gtrsim \Big(\min_{\substack{a\in \partial_*A\\b\in \partial_* B}}\tau_{\mathbb H}(a,b)\Big)^2\cdot n^{d-2} \cdot |\partial_*A| \cdot |\partial_*B|.
\end{equation} 
If $a\in \partial_*A$ (resp. $b\in\partial_*B$), then $a_1\geq \varepsilon n$ (resp. $b_1\geq \varepsilon n$). As a result, for $(a,b)\in \partial_*A\times \partial_*B$, one has $n\lesssim |a-b|\lesssim \min(a_1,b_1)$ (where the implicit constants depend on $\varepsilon$). This is where we crucially require that $\varepsilon>0$. We can therefore use \eqref{CH1} to conclude that
\begin{equation}
	\Big(\min_{\substack{a\in \partial_*A\\b\in \partial_* B}}\tau_{\mathbb H}(a,b)\Big)^2\gtrsim (n^{2-d})^2.
\end{equation}
Altogether, this shows that, with the notation of Lemma~\ref{Prop:regpoints}, 
\begin{align}\label{eq:last eq main bk} 
\begin{split}
\tau_{\mathbb H}(x,y) & \gtrsim  \frac 1{n^{d-2}} \sum_{A,B} \mathbb P\big[\mathcal C_n^e(x) = A, \mathcal C_n^e(y) = B\big] \cdot |\partial_*A| \cdot |\partial_*B| \\
& = \frac 1{n^{d-2}}\cdot \mathbb E\big[|\partial_* \mathcal C_n^e(x)|\big] \cdot \mathbb E\big[|\partial_* \mathcal C_n^e(y)|\big] \\&\gtrsim \frac 1{n^{d-2}}\cdot \mathbb E\big[X_n^{\varepsilon,K-\textrm{reg}}(x)\big]\cdot \mathbb E\big[X_n^{\varepsilon,K-\textrm{reg}}(y)\big]
\\&\gtrsim \frac{1}{n^{d-2}}\mathbb E[X_n^{\varepsilon}(x)]\cdot \mathbb E[X_n^\varepsilon(y)], 
\end{split}
\end{align}
where we used Lemma \ref{Prop:regpoints} in the last inequality. Using \eqref{eq: 2pt full space estimate} in \eqref{eq:last eq main bk} yields the existence of $c=c(\varepsilon)>0$ such that, for every $x,y\in \mathbb H$ satisfying $\lfloor |x-y|/3\rfloor\geq n_0$,
\begin{equation}
\tau_{\mathbb H}(x,y)\ge c \sum_{u\in \partial B_n^\varepsilon(x)} \sum_{v\in \partial B_n^\varepsilon(y)} 
\tau_{B_n(x)}(x,u)\cdot \tau_{\mathbb H}(u,v)\cdot \tau_{B_n(y)}(v,y).
\end{equation}
This concludes the proof.
\end{proof}

\subsection{Proof of Lemma~\ref{Prop:regpoints}}
\label{sec:proofPropreg}
The proof of Lemma \ref{Prop:regpoints} is very similar to the proofs of Theorem 4 in~\cite{KozmaNachmias} and Proposition 5.7 in~\cite{Asselah2025capacity}.  
First, one needs to introduce a local density condition. To be more precise, fix $n\ge 1$, $x\in \mathbb H$, and for $s>0$ and $z\in \partial B_n(x)$, consider the event 
\begin{align}
\begin{split}
\mathcal T_s^{\textrm{loc}}(z) =& \Big\{|\mathcal C\big(y; B_{s^d}(z)\cap B_n(x)\big)\cap B_s(z)|\leq s^4 (\log s)^4,\ \forall y \in B_s(z)\Big\} \\
				&\cap  \Big\{|\mathcal C(y;B_{s^d}(z)\cap B_n(x)) \cap B_s(z)\cap \partial B_n(x) |\leq s^2 (\log s)^4, \ \forall y\in B_s(z)\cap \partial B_n(x) \Big\} \\ 
				&\cap \Big\{ \exists  \text{ at most }  (\log s)^3 \text{ disjoint paths from }B_s(z) \text{ to } \partial B_{s^d}(z)\text{ in } B_n(x) \Big\}.	
				\end{split}		
\end{align}
The interest of this event, when compared to $\mathcal T_s(z)$, is that it only depends on the configuration of the percolation inside the box $B_{s^d}(z)$, and is thus a purely local event, while to determine whether $\mathcal T_s(z)$ holds or not, one needs a priori to know the configuration in the whole box $B_n(x)$. The drawback is that this event is a priori less likely than $\mathcal T_s(z)$, but as Lemma~\ref{lem:Tsloc} below shows, it is still extremely likely, and furthermore, the following simple fact holds by construction (see Claim 4.1 in~\cite{KozmaNachmias} or Claim 5.4 in~\cite{Asselah2025capacity}). 
\begin{Lem}\label{lem:Tsloc.Ts}
One has for every $n\ge 1$, every $x\in \mathbb H$, every $z\in \partial B_n(x)$, and every $s>0$, 
\begin{equation}
\mathcal T_s^{\textup{loc}}(z) \subseteq \mathcal T_s(z).
\end{equation} 
\end{Lem} 
As already mentioned, another fact we will use, and which is proved in~\cite[Claim~5.5]{Asselah2025capacity}, is the following. 
\begin{Lem}\label{lem:Tsloc}
There exists a constant $c>0$, such that for every $n\ge 1$, every $x\in \mathbb H$, every $z\in \partial B_n(x)$, and every $s>0$, 
\begin{equation}
\mathbb P[\mathcal T_s^{\textup{loc}}(z)] \ge 1 - \exp\big(-c(\log s)^4\big).
\end{equation}
\end{Lem} 
The last fact we shall need is Lemma 1.1 from~\cite{KozmaNachmias}, which we state here for the sake of completeness. 
\begin{Lem}\label{lem:KN}
There exist positive constants $c$ and $C$, such that for every $u\in \mathbb H$, every $s>0$,  and every $z\in \partial B_s(u)$,
\begin{equation} 
\tau_{B_s(u)}(u,v)\ge c\exp(-C(\log s)^2).
\end{equation}
\end{Lem} 
Using the FKG inequality, it follows from Lemma \ref{lem:KN}, that the probability to connect two arbitrary points of $\partial B_s(u)$ (for some $u\in \mathbb H$) is at least $c^2\exp(-2C(\log s)^2)$. We are now in a position to prove Lemma \ref{Prop:regpoints}. 

\begin{proof}[Proof of Lemma~\textup{\ref{Prop:regpoints}}] 
Fix $x\in \mathbb H$ and $\varepsilon \in (0,1)$. Let $n,K\ge 1$ to be fixed. Let us say that a point $z\in \partial B_n(x)$ is $s$-locally bad if the event $\mathcal T_s^{\textup{loc}}(z)$ does not hold, and let us denote by $X_n^{\varepsilon, s\textup{-loc-bad}}(x)$ the number of points on $\partial B_n^\varepsilon(x)$ which are 
$s$-locally bad and connected to $x$ in $B_n(x)$. 
Note that, due to Lemma~\ref{lem:Tsloc.Ts}, one has  
\begin{equation}\label{eq:Xns}
\mathbb E[X_n^{\varepsilon,K-\textrm{reg}}(x)] \ge \mathbb E[X_n^\varepsilon(x)] - \sum_{s\ge K} \mathbb E[X_n^{\varepsilon, s\textup{-loc-bad}}(x)].
\end{equation} 
We now upper bound each term of the above sum separately. 
First of all, for $s$ such that $8s^d > n$, we simply use Lemma~\ref{lem:Tsloc} and a union bound over all the points on $\partial B_n(x)$, to get (for some $c_1=c_1(d)>0$)
\begin{equation}
\mathbb E[X_n^{\varepsilon, s\textup{-loc-bad}}(x)]\lesssim n^{d-1}\cdot \exp(-c_1(\log n)^4).
\end{equation}
Additionally, for $s$ such that $s^2$ is much larger than $n^{d-1}$ (or larger than the total number of points on $\partial B_n(x)$), the set of $s$-locally bad points is empty by definition so that $\mathbb E[X_n^{\varepsilon, s\textup{-loc-bad}}(x)]=0$. On the other hand one has by Lemma~\ref{lem:KN} (and the remark following it),
\begin{equation}
	\mathbb E[X_n^\varepsilon(x)] \gtrsim \exp(-C_1(\log n)^2),
\end{equation}
for some constant $C_1>0$. Combining the two previously displayed equations, one can choose $n$ large enough so that 
\begin{equation}\label{eq:Xns2}
\sum_{s:8s^d>n} \mathbb E[X_n^{\varepsilon, s\textup{-loc-bad}}(x)]\le \frac 14 \cdot \mathbb E[X_n^\varepsilon(x)]. 
\end{equation}
We now fix $s\ge K$ such that $8s^d\le n$ and consider the set $U = \{ u \in \mathbb Z^d : u_i \in \{0,s^d\}, \ \forall i =1,\dots,d\}$. For each $w\in U$, we define 
\begin{equation}
\mathcal B(w) = \{B_{2s^d}(z) : z\in w+4s^d\cdot \mathbb Z^d\}.
\end{equation} 
Denote by $Q(w)$ the set of all the boxes of this partition which intersect $\partial B_n^\varepsilon(x)$.
We explore them using the following algorithm. First, we reveal the whole percolation configuration outside the union of these boxes. 
Next, if at least one of the boxes of $Q(w)$ is connected to $x$ by an open path, that remains in the explored region, we choose one at random and 
reveal the configuration inside it, and we continue as long as there still exist an unexplored box of $Q(w)$ which is connected to $x$ via an open path in the explored region.  
We let $N(w)$ be the number of boxes of $Q(w)$ which have been revealed by this algorithm. Note that during this exploration procedure, each time we reveal the configuration inside a 
new box, almost surely, conditionally on the configuration outside the box, the probability that it contains a pioneer point is at least $c_2\exp(-C_2 (\log s)^2)$ by Lemma~\ref{lem:KN} and the remark following it, for some constant $c_2,C_2>0$. Hence, for any $w\in U$, 
\begin{equation}
\mathbb E[X_n^\varepsilon(x)] \ge c_2\exp(-C_2(\log s)^2) \cdot \mathbb E[N(w)], 
\end{equation} 
and hence also (since $|U|=2^d$), 
\begin{equation}\label{eq:Xneps}
\mathbb E[X_n^\varepsilon(x)] \ge \frac{c_2}{2^d} \exp(-C_2(\log s)^2) \cdot \sum_{w\in U} \mathbb E[N(w)]. 
\end{equation}
For a box $q\in Q(w)$, call the \emph{interior} of $q$ the set of points in $q$ which are at distance at least $s^d$ from the points which are in $q^c\cap B_n(x)$. Observe that as $w$ varies in $U$, the union of all the interiors of the boxes $q\in Q(w)$ covers the whole boundary $\partial B_n^\varepsilon(x)$ (recall that we assume $8s^d\le n$).   
Note also that for a point $z$ on $\partial B_n^\varepsilon(x)$ which is in the interior of a box $q\in Q(w)$, the event $\mathcal T_s^{\textup{loc}}(z)$ only depends on the configuration of edges inside $q$. Since there are at most order $s^{d^2}$ such points in each box $q\in Q(w)$, a union bound and Lemma~\ref{lem:Tsloc} give that for some constants $c_3,C_3>0$, and for any $s$ as above,
\begin{equation}  
\mathbb E[X_n^{\varepsilon, s\textup{-loc-bad}}(x)] \le C_3 s^{d^2}\exp(-c_3(\log s)^4) \sum_{w\in U}  \mathbb E[N(w)].
\end{equation}
Hence, using again~\eqref{eq:Xneps}, and taking $K$ large enough ensures that for all $n\ge 1$,
\begin{equation}
\sum_{s:K\le s \le (n/8)^{1/d}}\mathbb E[X_n^{\varepsilon, s\textup{-loc-bad}}(x)] \le \frac 14 \cdot \mathbb E[X_n^\varepsilon(x)].
\end{equation}
Together with~\eqref{eq:Xns} and~\eqref{eq:Xns2}, this concludes the proof of the lemma. 
\end{proof} 

\paragraph{Acknowledgements.} RP thanks Hugo Duminil-Copin for stimulating discussions at an early stage of this project, and acknowledges the support of the Swiss National Science Foundation through a Postdoc.Mobility grant. BS acknowledges the support from the grant ANR-22-CE40-0012 (project LOCAL).

\bibliographystyle{alpha}
\bibliography{biblio.bib}

\begin{thebibliography}{DCP25b}

\bibitem[ASS25]{Asselah2025capacity}
Amine Asselah, Bruno Schapira, and Perla Sousi.
\newblock Capacity in high dimensional percolation.
\newblock {\em Preprint}, 2025.
\newblock \url{https://arxiv.org/pdf/2509.21253}.

\bibitem[BS85]{BrydgesSpencerSAW}
David Brydges and Thomas Spencer.
\newblock Self-avoiding walk in 5 or more dimensions.
\newblock {\em Communications in Mathematical Physics},
  \textbf{97}(1):125--148, 1985.

\bibitem[CH20]{ChatterjeeHanson}
Shirshendu Chatterjee and Jack Hanson.
\newblock Restricted percolation critical exponents in high dimensions.
\newblock {\em Communications on Pure and Applied Mathematics},
  \textbf{73}(11):2370--2429, 2020.

\bibitem[CHS23]{ChatterjeeHansonSosoe2023subcritical}
Shirshendu Chatterjee, Jack Hanson, and Philippe Sosoe.
\newblock Subcritical connectivity and some exact tail exponents in high
  dimensional percolation.
\newblock {\em Communications in Mathematical Physics},
  \textbf{403}(1):83--153, 2023.

\bibitem[DCP25a]{DumPan24Perco}
Hugo Duminil-Copin and Romain Panis.
\newblock An alternative approach for the mean-field behaviour of spread-out
  {B}ernoulli percolation in dimensions $d>6$.
\newblock {\em Probability Theory and Related Fields}, 2025.

\bibitem[DCP25b]{DumPan24WSAW}
Hugo Duminil-Copin and Romain Panis.
\newblock An alternative approach for the mean-field behaviour of weakly
  self-avoiding walks in dimensions $d>4$.
\newblock {\em Probability Theory and Related Fields}, 2025.

\bibitem[DCT16]{DuminilTassionNewProofSharpness2016}
Hugo Duminil-Copin and Vincent Tassion.
\newblock A new proof of the sharpness of the phase transition for {B}ernoulli
  percolation and the {I}sing model.
\newblock {\em Communications in Mathematical Physics}, \textbf{343}:725--745,
  2016.

\bibitem[FvdH17]{FitznervdHofstad2017Percod10}
Robert Fitzner and Remco~W. van~der Hofstad.
\newblock Mean-field behavior for nearest-neighbor percolation in $d>10$.
\newblock {\em Electronic Journal of Probability}, \textbf{22}:43, 2017.

\bibitem[Gri99]{GrimmettPercolation1999}
Geoffrey Grimmett.
\newblock {\em Percolation}, volume \textbf{321}.
\newblock Springer, 1999.

\bibitem[Har08]{HaraDecayOfCorrelationsInVariousModels2008}
Takashi Hara.
\newblock Decay of correlations in nearest-neighbor self-avoiding walk,
  percolation, lattice trees and animals.
\newblock {\em The Annals of Probability}, \textbf{36}(2):530--593, 2008.

\bibitem[HHS03]{HaraSladevdHofstad2003PercoSO}
Takashi Hara, Remco van~der Hofstad, and Gordon Slade.
\newblock Critical two-point functions and the lace expansion for spread-out
  high-dimensional percolation and related models.
\newblock {\em The Annals of Probability}, \textbf{31}(1):349--408, 2003.

\bibitem[HMS23]{HutchcroftMichtaSladePercolationTorusPlateau2023}
Tom Hutchcroft, Emmanuel Michta, and Gordon Slade.
\newblock High-dimensional near-critical percolation and the torus plateau.
\newblock {\em The Annals of Probability}, \textbf{51}(2):580--625, 2023.

\bibitem[HS90]{HaraSlade1990Perco}
Takashi Hara and Gordon Slade.
\newblock Mean-field critical behaviour for percolation in high dimensions.
\newblock {\em Communications in Mathematical Physics},
  \textbf{128}(2):333--391, 1990.

\bibitem[HS14]{vdHofstadSapoS14}
Remco van~der Hofstad and Artem Sapozhnikov.
\newblock Cycle structure of percolation on high-dimensional tori.
\newblock {\em Annales de l'Institut Henri Poincaré: Probabilités et
  Statistiques}, {\bf 50}:999--1027, (2014).

\bibitem[Hut25]{hutchcroft2025dimension}
Tom Hutchcroft.
\newblock Dimension dependence of critical phenomena in long-range percolation.
\newblock {\em Preprint}, 2025.
\newblock \url{https://arxiv.org/pdf/2510.03951}.

\bibitem[KN11]{KozmaNachmias}
Gady Kozma and Asaf Nachmias.
\newblock Arm exponents in high dimensional percolation.
\newblock {\em Journal of the American Mathematical Society},
  \textbf{24}(2):375--409, 2011.

\bibitem[LL10]{LawlerLimicRandomWalks2010}
Gregory~F. Lawler and Vlada Limic.
\newblock {\em Random walk: a modern introduction}, volume \textbf{123}.
\newblock Cambridge University Press, 2010.

\bibitem[Pan24]{panis2024applications}
Romain Panis.
\newblock {\em Applications of path expansions to statistical mechanics}.
\newblock PhD thesis, PhD thesis, University of Geneva, 2024.

\bibitem[Sla06]{SladeSaintFlourLaceExpansion2006}
Gordon Slade.
\newblock {\em The {L}ace {E}xpansion and {I}ts {A}pplications: {E}cole
  D'{E}t{\'e} de {P}robabilit{\'e}s de {S}aint-{F}lour {XXXIV}-2004}.
\newblock Springer, 2006.

\end{thebibliography}
\end{document}